\documentclass[12pt]{amsart}
\usepackage{amssymb}
\usepackage{amsmath,amscd}
\usepackage{amsfonts,latexsym,verbatim,amscd,mathrsfs,color,array}
\usepackage[all,cmtip]{xy}
\usepackage{mathrsfs}
\usepackage{multirow}
\usepackage{graphicx}
\usepackage{amsfonts, amsthm}
\usepackage{bbm}
\usepackage[hmargin=1in,vmargin=1in]{geometry}
\usepackage{mathdots}
\usepackage{stmaryrd}
\usepackage{MnSymbol}
\usepackage{bbm}
\usepackage[hidelinks]{hyperref}

\allowdisplaybreaks

\def\XXint#1#2#3{{\setbox0=\hbox{$#1{#2#3}{\int}$ }
		\vcenter{\hbox{$#2#3$ }}\kern-.6\wd0}}

\DeclareMathSymbol{\intprod}{\mathbin}{MnSyC}{'270}
\newcommand{\LB}{\left[}
\newcommand{\RB}{\right]}
\newcommand{\LA}{\left\langle}
\newcommand{\RA}{\right\rangle}

\newcommand{\T}{{\mathcal T}}

\newcommand{\dist}{{\mbox{dist}}}

\newcommand{\pd}{{\partial}}

\newcommand{\inj}{{\mathrm{inj} }}
\newtheorem{thm}{Theorem}[section]
\newtheorem{lemma}[thm]{Lemma}
\newtheorem*{lemma*}{Lemma}
\newtheorem{prop}[thm]{Proposition}

\newtheorem{cor}[thm]{Corollary}

\newtheorem*{conj*}{Conjecture}

   \newtheoremstyle{others}
     {3pt}
     {2pt}
     {}
     {}
     {\bf}
     {.}
     {.5em}
     {}

\theoremstyle{others}

\newtheorem*{rmk*}{Remark}
\newtheorem{defn}[thm]{Definition}

\numberwithin{equation}{section}

 \DeclareMathOperator{\tr}{tr}
 
 \def\<{\left\langle} \def\>{\right\rangle}
 \def\({\left(} \def\){\right)}

 \newcommand{\eps}{\varepsilon}

 \newcommand{\Ga}{\Gamma}
 
 \renewcommand{\th}{\theta}

 \newcommand{\p}{\partial}
 \newcommand{\n}{\nabla}

\begin{document}

\title[Harmonic map flow with strict type-II blowup]{Strict type-II blowup in harmonic map flow}
\author{Alex Waldron}
\address{University of Wisconsin, Madison}
\email{waldron@math.wisc.edu}

\begin{abstract} A finite-time singularity of 2D harmonic map flow will be called ``strictly type-II'' if the outer energy scale satisfies
$$\lambda(t) = O (T-t)^{\frac{1 + \alpha}{2} }.$$
We prove that the body map at a strict type-II blowup is H\"older continuous. This is relevant to a conjecture of Topping. 
\end{abstract}

\maketitle

\thispagestyle{empty}


\section{Introduction}

Let $M$ and $N$ be compact Riemannian manifolds. For differentiable maps $u: M \to N,$ we may define the Dirichlet energy:
$$\frac12 \int_M |d u|^2 \, dV.$$
Its downward gradient flow is given by
\begin{equation}\label{hm}
\frac{\pd u}{\pd t} = \mathcal{T}(u).
\end{equation}
Here, $\T(u)$ is the tension field of $u,$ a generalization of the Laplace-Beltrami operator to maps between manifolds.
The evolution equation (\ref{hm}), known as \emph{harmonic map flow}, was introduced in 1964 by Eells and Sampson \cite{eellssampson} and has been studied since then almost without interruption.


When $M$ has dimension two, the Dirichlet functional is conformally invariant; we shall be concerned exclusively with this case. Struwe \cite{struwehmsurfaces} constructed a global weak solution $u(t)$ 
of (\ref{hm}) starting from any initial map of Sobolev class $W^{1,2}(M,N).$ The solution is smooth away from finitely many singular times, when isolated singularities (``bubbles'') may form. For a singular time $T < \infty,$ the limit
$$u(T) = \lim_{t \nearrow T} u(t)$$
exists weakly in $W^{1,2}$ and smoothly away from the bubbling set, 
and is referred to as the \emph{body map}. 

Note that Struwe's construction leaves open the possibility that the body map will be discontinuous.
Topping \cite{toppingwindingbehavior} demonstrated that for certain target manifolds and initial data, $u(T)$ can indeed have an essential singularity at a bubble point.
At the same time, he conjectured that for well-behaved (specifically, real-analytic) metrics on the target,
the body map always extends continuously across the bubbling set. Topping's conjecture is the {\it sine qua non} for future geometric applications of harmonic map flow.

In previous joint work with C. Song \cite{songwaldron}, we established H\"older continuity of the body map when $N$ is compact K\"ahler with nonnegative holomorphic bisectional curvature and the energy of the initial map is near the holomorphic energy. 
The argument relied partly on establishing a bound on the outer energy scale of the form
\begin{equation}\label{stricttypeII}
\lambda(t) = O(T - t)^{\frac{1 + \alpha}{2}},
\end{equation}
where $0 < \alpha \leq 1.$
We refer to a singularity satisfying (\ref{stricttypeII})
as a \emph{strict type-II blowup}. Note that (\ref{stricttypeII}) is a refinement of the ordinary type-II blowup estimate for 2D harmonic map flow: 
\begin{equation}\label{typeII}
\lambda(t) = o(T - t)^{\frac12}.
\end{equation}
For a proof of (\ref{typeII}), see 
\cite[Theorem 1.6v]{toppingwindingbehavior}. 

The strict type-II bound (with $\alpha = 1$) is most familiar from the rotationally symmetric setting. Angenent, Hulshof, and Matano \cite{angenenthulshofmatano} proved that the finite-time blowup first constructed by Chang, Ding, and Ye \cite{changdingye} occurs with rate $\lambda(t) = o(T - t).$ 
Raphael and Schweyer \cite{raphaelschweyer} determined a large set of rotationally symmetric initial data that blows up under (\ref{hm}) with the precise rate
\begin{equation}\label{raphaelschweyerblowuprate}
\lambda(t) \sim \kappa \frac{T - t}{(\ln T - t )^2}.
\end{equation}
They also proved in this context that the body map is $W^{2,2},$ hence $C^\beta$ for each $\beta < 1.$\footnote{This also follows from the main theorem of \cite{songwaldron}.} Davila, Del Pino, and Wei \cite{daviladelpinowei} produced a larger set of examples with blowup rate (\ref{raphaelschweyerblowuprate}), whose body maps are continuous by construction. 

In another direction, Topping \cite{toppingholder} proved continuity of the body map if the Dirichlet energy is H\"older continuous as a function of time. One step in the proof was to establish that the strict type-II bound (\ref{stricttypeII}) follows from this assumption (see \cite[Lemma 2.2]{toppingholder}). Hence, the strict type-II bound with arbitrary exponent has appeared in previous work, although a much stronger assumption was required to obtain continuity of the body map.


Accordingly, the known examples of strict type-II blowup all have continuous body maps. On the other hand, the only known example with discontinuous body map, due to Topping, fails to be strictly type-II---see \cite[Theorem 1.14e]{toppingwindingbehavior}. Our main theorem confirms the implication, as follows. 

\begin{thm}\label{thm:intro} For any strict type-II blowup of harmonic map flow in dimension two, with $0 < \alpha \leq 1$ in (\ref{stricttypeII}),
the body map is $C^{ \frac{\alpha}{3} }.$
\end{thm}


\noindent Our main technical result is Theorem \ref{thm:mainthm} below, 
and Corollary \ref{cor:detailedthm} gives the formal statement of Theorem \ref{thm:intro}. 


The proof 
depends on obtaining decay estimates for the differential of $u$ in the ``neck region,'' {\it i.e.}, the area near the singularity but outside the (vanishing) energy scale.
The required estimate on the angular component of $du$ is already known \cite[Lemma 5.4]{songwaldron}, 
so it remains only to estimate the radial component.
We obtain a bound on the difference between the radial and angular components from a well-known identity (\ref{stressidentity}), giving an integral bound on the radial component under the flow. 
Using a specialized parabolic estimate (Proposition \ref{gestimate}) and a bootstrap argument, we are able to promote this integral bound to pointwise decay throughout the neck region. 

\subsection*{Acknowledgement} The author is partially supported by NSF DMS-2004661.

\section{Technical results}

For an introduction to harmonic map flow, we refer the reader to \cite[\S 2]{songwaldron} or to the textbook of Lin and Wang \cite{linwangbook}.

Let $(M,g)$ and $(N,h)$ be Riemannian manifolds.
Given any smooth map $u : M \to N,$ we denote the pullback of $h$ to $u^*TN$ by $\LA \cdot, \cdot \RA,$ which we combine with $g$ on tensors.
The differential $du$ is a section of $T^*M \otimes u^*TN,$ with norm squared
$$|du|^2 = g^{ij} \LA \p_i u, \p_j u \RA.$$
The tension field is given by
$$\T(u) = \tr_g \n du.$$
Here, $\nabla$ is the Levi-Civita connection on $T^*M$ coupled with the pullback to $u^*TN$ of the Levi-Civita connection on $N.$

Suppose that $\dim(M) = 2.$ 
We write
\begin{equation}\label{conformalmetric}
g=\xi^2(dr^2+r^2d\th^2)
\end{equation}
for the given metric in conformal coordinates, where $\xi$ is a smooth function. By adjusting the conformal chart to second order, we may assume $d \xi (0) = 0.$ We may further assume
\begin{equation}\label{smallness1}
\left| \xi - 1\right| + r|d \xi| + r^2 |\nabla d \xi| \leq \xi_0 r^2
\end{equation}
for a constant $0 \leq \xi_0 \leq \frac12,$ after rescaling $g$ by a constant. This implies
\begin{equation}\label{rversusdist}
|r- \dist_g(x_0, \cdot)| \leq C \xi_0 r^2.
\end{equation}
In view of these bounds, the difference between conformal and geodesic coordinates will be of no significance for our results; we use conformal coordinates only for convenience.

Let
$$S_{ij} = \LA \p_i u , \p_j u \RA - \frac12 g_{ij} |du|^2$$
be the \emph{stress-energy tensor} of $u.$ This is a symmetric 2-tensor on $M,$ which satisfies
\begin{equation}\label{stressdiv0}
\nabla^i S_{ij} = \LA \T(u), \p_j u \RA.
\end{equation}
For a derivation of (\ref{stressdiv0}), see \cite[\S 2.2]{songwaldron}.

The radial vector field $X^i$ in conformal coordinates is conformally Killing, {\it i.e.}
\begin{equation*}
\nabla^i X^j + \n^j X^i = \lambda g^{ij}
\end{equation*}
for a scalar function $\lambda.$ Contracting with $X^i$ in (\ref{stressdiv0}), we have
\begin{equation}\label{stressdiv}
\begin{split}
\nabla^i (X^j S_{ij}) & = \LA \T(u), X^j \p_j u \RA + \frac12 \left( \nabla^i X^j + \n^j X^i \right) S_{ij} \\
& =  \LA \T(u), X^j \p_j u \RA,
\end{split}
\end{equation}
since $g^{ij}S_{ij} = 0$ in dimension two.
Integrating (\ref{stressdiv}) over a disk $D_r$ in the conformal chart, and applying the divergence theorem, we obtain
\begin{equation}\label{stressidentity}
\int_{S^1_r} X^i X^j S_{ij} \, d\theta = \int_{D_r} \LA \T(u), X^j \p_j u \RA \, dV. 
\end{equation}
This identity is well known from the theory of approximate harmonic maps, and will be used crucially below.


Next, we need the following parabolic estimates. For $0 \leq \nu \leq 1$ and $\mu > 1,$ let
$$\square_\nu = \p_t - \left( \p_r^2 + \frac{1}{r} \p_r - \frac{\nu^2}{r^2} \right),$$
$$\Delta_\mu =\p_r^2 + \frac{(\mu - 1)}{r} \p_r.$$
We have
\begin{equation}\label{squaretriangle}
\square_\nu \left( r^{\nu} y \right) = r^\nu \left( \p_t - \Delta_{2(\nu + 1)} \right) y.
\end{equation} 
Also notice that
\begin{equation}\label{squarerbeta}
\square_\nu \! r^\beta = \left( \nu^2 - \beta^2 \right) r^{\beta - 2}.
\end{equation}

\begin{lemma}\label{lemma:evolutions}
Let $u$ be a solution of (\ref{hm}) 
with respect to a metric $g$ of the form (\ref{conformalmetric}-\ref{smallness1}). Suppose that for some $0 < \eta^2 < \eps_0,$ 
 we have
\begin{equation}\label{evolutions:derivbounds}
r|du|+ r^2|\n du| + r^3|\n^2 du| + r^4|\n^3 du| \le \eta. 
\end{equation}
Then the angular energy
\begin{equation*}
f = f(u;r, t) := \sqrt{ \int_{ S_r^1 } |u_\th (r,\theta,t)|^2 d\th + \int_{ S_r^1 } |\nabla_\theta u_\th (r,\theta,t)|^2 d\th }
\end{equation*}
satisfies a differential inequality
\begin{equation}\label{evolutions:angular}
\square_\nu \! f \le C \xi_0 \eta,
\end{equation}
where $\nu = \sqrt{1 - C \eta}.$
The radial energy
\begin{equation}\label{radialenergy}
g = g(u;r, t) := \sqrt{ \int_{ S_r^1 } r^2 |u_r (r,\theta,t)|^2 d\th }
\end{equation}
satisfies
\begin{equation}\label{evolutions:radial}
\square_\nu \! \left(\frac{g}{r} \right) \le \frac{6f}{r^3} + \frac{ C \xi_0 \eta}{r}.
\end{equation}
Here, $\varepsilon_0$ depends on the geometry of $N,$ and $\xi_0$ is the constant of (\ref{smallness1}).
\end{lemma}
\begin{proof} The proof is an elementary extension of prior calculations; see Appendix \ref{app:evolutions}.
\end{proof}


\begin{prop}\label{festimate} Let $-\nu \leq \beta_i \leq \nu \leq 1,$ for $i = 0,1.$ 
Suppose that $f(r,t)$ is continuous on $\LB \rho, 1 \RB \times \LB \tau, T \right)$ and satisfies
\begin{equation}\label{festimate:equation}
\Box_\nu \! f \leq A,
\end{equation}
with
\begin{equation}\label{festimate:assumptions}
\begin{split}
| f(r,\tau) | & \leq A \left( \left( \frac{\rho}{r} \right)^{\beta_0} + r^{\beta_1} \right), \\
f(\rho, t)  \leq & A, \qquad f(1, t)  \leq A.
\end{split}
\end{equation}
Then, given $0 < \kappa \leq 1/2,$ 
for
$$\frac{\rho}{\kappa} \leq r \leq \kappa \quad \text{and} \quad \tau + \frac{r^2}{\kappa^2} \leq t < T,$$
we have
\begin{equation}\label{festimate:conclusion}
| f(r,t) | \leq C_{\ref{festimate}} A \left( \kappa^{\nu + \beta_0} \left( \frac{\rho}{r} \right)^{\beta_0} + \kappa^{\nu - \beta_1} r^{\beta_1} \right).
\end{equation}
Here $C_{\ref{festimate}}$ depends on $\beta_0, \beta_1,$ and $\nu.$
\end{prop}
\begin{proof} Let $\mu = 2 \nu + 2.$ Using the results of Appendix \ref{app:initialkernel} and (\ref{squaretriangle}-\ref{squarerbeta}), we can construct a supersolution for (\ref{festimate:equation}) of the form
$$\bar{v}(r,t) = r^\nu  v_0(r,t) + 2 A \left( \left( \frac{\rho}{r} \right)^\nu + r^\nu \right) - \frac{A}{4 - \nu^2} r^{2},$$
where
$$\left( \p_t - \Delta_{\mu} \right) v_0 = 0,$$
$$v_0(r,0) = r^{-\nu} f(r,0), \qquad v_0(\rho,t) = 0 = v_0(1,t).$$
Applying the comparison principle to (\ref{festimate:equation}), we have
\begin{equation}\label{festimate:maxprin}
g(r,t) \leq \bar{v}(r,t),
\end{equation}
so it suffices to check (\ref{festimate:conclusion}) for $\bar{v}.$

By Proposition \ref{initialestimate} and (\ref{festimate:assumptions}), we have
\begin{equation*}
v_0(r,t) \leq C A \left( \frac{\rho^{\beta_0} }{r^{\beta_0 + \nu} }w^{\beta_0 + \nu} \left( r,t - \tau\right) + r^{\beta_1 - \nu} w^{\nu - \beta_1 }(r,t - \tau) \right),
\end{equation*}
where $w^a(r,t)$ is defined by
$$w^a(r,t) = \left( \frac{r^2}{r^2 + t} \right)^{a/2}.$$
Overall, from (\ref{festimate:maxprin}), we obtain
\begin{equation}\label{festimate:flongest}
| f(r,t) | \leq CA \left( \left( \frac{\rho}{r} \right)^{\beta_0} w^{\beta_0 + \nu}(r,t - \tau) + r^{\beta_1} w^{\nu - \beta_1}(r,t - \tau) + \left( \frac{\rho}{r} \right)^\nu  + r^\nu + r^{2} \right).
\end{equation}
For $r \geq \rho/ \kappa,$ we have
$$\left( \frac{\rho}{r} \right)^\nu \leq \kappa^{\nu - \beta_0} \left( \frac{\rho}{r} \right)^{\beta_0}.$$
For $r \leq \kappa$ and $t - \tau \geq r^2 / \kappa^2,$ we have
$$ w^{\beta_0 + \nu}(r,t - \tau) \leq \kappa^{\beta_0 + \nu} , \qquad w^{\nu - \beta_1}(r,t - \tau) \leq \kappa^{\nu - \beta_1}, \qquad r^\nu \leq \kappa^{\nu - \beta_1} r^{\beta_1}.$$
Substituting into (\ref{festimate:flongest}), we obtain (\ref{festimate:conclusion}).
\end{proof}

\begin{prop}\label{gestimate} Let $- \nu \leq \gamma_i \leq \nu \leq 1$ and $\beta_i$ with $|1 \pm \beta_i| \neq \nu,$ for $i = 0,1.$ Given $0 < \kappa \leq 1/2,$ let $\rho \leq \rho_1 \leq \kappa.$

Suppose that $g(r,t)$ is continuous on $\LB \rho_1, 1 \RB \times \LB \tau, T \right)$ and satisfies
\begin{equation}\label{gestimate:equation}
\square_\nu \! \left( \frac{g}{r} \right) \leq \frac{A}{r^3} \left( \left(\frac{\rho}{r}\right)^{\beta_0} + r^{\beta_1} \right),
\end{equation}
with
\begin{equation}\label{gestimate:assumptions}
\begin{split}
|g(r,\tau)| & \leq A \left( \left( \frac{\rho}{r}\right)^{\gamma_0} + r^{\gamma_1} \right), \\
\int_{\tau}^{T} \!\! g(\rho_1, t)^2\, dt & \leq B^2, \qquad | g(1,t) | \leq A.
\end{split}
\end{equation}
Then, 
for
\begin{equation}\label{gestimate:rtassumptions}
2 \rho_1 \leq r \leq \kappa \quad \text{and} \quad \tau + \frac{r^2}{\kappa^2} \leq t < T,
\end{equation}
we have
\begin{equation}\label{gestimate:estimate}
| g(r,t) | \leq C_{\ref{gestimate}} \left( B \frac{\rho_1^{\nu - 1}}{r^\nu}+ A \left( \kappa^{\nu + \gamma_0 + 1}  \left( \frac{\rho}{r} \right)^{\gamma_0} + \kappa^{\nu - \gamma_1 + 1} r^{\gamma_1} + \left( \frac{\rho}{r} \right)^{\beta_0} + r^{\beta_1} \right)\right).
\end{equation}
Here, $C_{\ref{gestimate}}$ depends on $\gamma_0, \gamma_1, \beta_0, \beta_1,$ and $\nu.$
\end{prop}
\begin{proof} As in the previous proof, we construct a supersolution for (\ref{gestimate:equation}) of the form
$$\bar{v}(r,t) = r^\nu \left( v_1(r,t) + v_2(r,t) \right) + C A \left( \frac{\rho^{\beta_0}}{r^{\beta_0 + 1}} + r^{\beta_1 - 1} \right),$$
where
$$\left( \p_t - \Delta_\mu \right) v_i = 0, \quad i = 1,2,$$
with
$$v_1(r,0) = r^{-\nu - 1} g(r,0), \qquad v_1(\rho,t) = 0 = v_1(1,t),$$
and
$$v_2(\rho_1, t) = \rho_1^{-\nu - 1} g(\rho_1, t), \qquad v_2(r,0) = 0= v_2(1,t).$$
Applying the comparison principle to (\ref{gestimate:equation}), we have
\begin{equation}\label{gestimate:maxprin}
g(r,t) \leq r \bar{v}(r,t).
\end{equation}
By Proposition \ref{innerestimate} and (\ref{gestimate:assumptions}), we have
\begin{equation*}
v_1(r,t) \leq C A \left( \frac{\rho_1^{\gamma_0}}{r^{\gamma_0 + \nu + 1} }w^{\gamma_0 + \nu + 1} \left( r,t \right) + r^{\gamma_1 - \nu - 1} w^{\nu - \gamma_1 + 1}(r,t) \right).
\end{equation*}
By Proposition \ref{innerestimate} and (\ref{gestimate:assumptions}), since $\mu - 2 = 2\nu,$ we have
\begin{equation*}
v_2(r,t) \leq C B \rho_1^{-\nu - 1} \left( \frac{\rho_1^{2 \nu}}{r^{2 \nu + 1}} \right) \leq CB \frac{\rho_1^{\nu - 1}}{r^{2\nu + 1}}.
\end{equation*}
Overall, from (\ref{gestimate:maxprin}), we obtain
\begin{equation*}
| g(r,t) | \leq C\left( \frac{B \rho_1^{\nu - 1}}{r^\nu} + A \left( \left( \frac{\rho_1}{r} \right)^{\gamma_0} w^{\gamma_0 + \nu + 1}(r,t) +  r^{\gamma_1} w^{\nu - \gamma_1 + 1}(r,t) + \left( \frac{\rho}{r} \right)^{\beta_0} + r^{\beta_1} \right) \right).
\end{equation*}
As in the previous proof, under the assumptions (\ref{gestimate:rtassumptions}) on $r$ and $t,$ this implies (\ref{gestimate:estimate}).
\end{proof}

\section{Main Theorem}\label{sec:mainthm}

Fix $x_0 \in M$ and let $r_g(x) = \dist_g(x_0, x).$ We shall denote the geodesic ball of radius $R$ centered at $x_0$ by $B_R(x_0),$ or more typically by $B_R.$ An annulus will be denoted by $U_\rho^R = B_R \setminus \bar{B}_\rho.$

\begin{defn}[\cite{songwaldron}, Definition 4.2] Given a $W^{1,2}$ map $u: B_R(x_0) \to N,$ the \emph{outer energy scale} $\lambda_{\eps, R, x_0}(u)$ is the smallest nonnegative number $\lambda$ such that
\begin{equation}\label{energyscale}
\sup_{\lambda < \rho < R} \int_{U^{\rho}_{\rho/2}(x_0)} |du|^2 \, dV < \eps.
\end{equation}
Note that $\lambda = R$ satisfies (\ref{energyscale}) vacuously, so $0 \leq \lambda_{\eps, R, x_0}(u) \leq R$ by definition.
\end{defn}
We first establish the following ``baby case'' of the main theorem. 

\begin{lemma}\label{basecase} Given $E, \lambda> 0$ and $0 < \eps < \eps_0,$ there exists $\delta > 0$ as follows. Suppose that $u$ is a smooth solution of (\ref{hm}) on $B_R \times ( - R^2, T),$ with $0 < R < R_0$ and $T > 0,$ which satisfies
\begin{equation}\label{basecase:energyassumption}
\sup_{-R^2<t<T} \int_{B_R} |du(t)|^2 \, dV \leq E,
\end{equation}
\begin{equation}\label{basecase:deltaassumption}
\int_{-R^2}^T \!\! \int_{B_R} |\T|^2 \, dV dt < \delta^2,
\end{equation}
and
\begin{equation}\label{basecase:lambdaassumption}
\sup_{-R^2<t<T} \lambda_{\eps, R, x_0}(u(t)) \leq R \lambda.
\end{equation}
Then for $2 R \lambda \leq r_g(x) \leq R/2$ and each integer $k \geq 0,$ we have
\begin{equation}\label{basecase:est}
r_g(x)^{1 + k} |\nabla^{(k)} du(x,0)| \leq C_k \sqrt{\eps} \left( \frac{R \lambda}{r_g(x)} + \frac{r_g(x)}{R} \right).
\end{equation}
Here, $R_0 > 0$ depends on the geometry of $M,$ $\eps_0 > 0$ depends on the geometry of $N,$ and $\delta$ depends on $E,\eps,$ and $\lambda.$
\end{lemma}
\begin{proof} 
Assuming that $R_0$ is sufficiently small, we may rescale so that $R = 1$ and our metric $g$ takes the form (\ref{conformalmetric}-\ref{smallness1}), with $\xi_0 \leq \frac12.$ By (\ref{rversusdist}), it suffices to establish (\ref{basecase:est}) with the conformal coordinate $r$ in place of $r_g.$

In view of (\ref{basecase:lambdaassumption}), the standard $\eps$-regularity lemma (see {\it e.g.} Theorem 3.4 of \cite{songwaldron}) implies
\begin{equation}\label{basecase:crudestimate}
\sup_{\lambda \leq r \leq R} \left( r |du| + r^2 |\nabla du| + r^3 |\nabla^2 du | + r^4 |\nabla^3 du | \right) \leq C \sqrt{\eps}.
\end{equation}
Hence, by Lemma \ref{lemma:evolutions}, $f$ satisfies an evolution equation
$$\square_\nu f \leq C \sqrt{\eps},$$
where $\nu = \sqrt{1 - C \sqrt{\eps}}.$
We can apply Proposition \ref{festimate}  with $\rho = \lambda,$ $A = C \sqrt{\eps},$ $\beta_0 = \beta_1 = 0,$ and $\tau = -\frac12.$ Since $w^a(r, s) \leq C r^a$ for $s \geq \frac14,$ (\ref{festimate:flongest}) gives
\begin{equation}\label{basecase:improvedfestimate}
\begin{split}
\sup_{-\frac14 \leq t \leq 0} | f(r,t) | & \leq C\sqrt{\eps} \left( r^{\nu} + r^{\beta_1} r^{\nu - \beta_1} + \left( \frac{\lambda}{r} \right)^\nu + r^\nu + r^{2} \right) \\
& \leq C \sqrt{\eps} \left( \left( \frac{\lambda}{r} \right)^\nu  + r^\nu \right).
\end{split}
\end{equation}
Applying H\"older's inequality identity to (\ref{stressidentity}), we have
\begin{equation}\label{basecase:holder}
\begin{split}
\left| \int_{-\frac12}^{0} \left( g^2(r, t) - f^2(r, t) \right) \, dt \right| 
& = \left| \int_{t_0}^{t_1} \!\!\!\!\! \int_{S^1_{\rho_1}} X^i X^j S_{ij} (x,t) \, d\theta dt \right| \\
& \leq \left( \int_{t_0}^{t_1} \!\!\!\!\! \int_{D_{\rho_1}} |\T(u)|^2 \, dV dt  \right)^{\frac12} \left( \int_{t_0}^{t_1} \!\!\!\!\!  \int_{D_{\rho_1}} r'^2 \, |du|^2 \, dV_{r'} dt \right)^{\frac12} \\
& \leq C \delta \sqrt{E}.
\end{split}
\end{equation}
Since $|du|^2 \leq f^2 + g^2,$ for $\delta$ sufficiently small, (\ref{basecase:improvedfestimate}-\ref{basecase:holder}) imply
\begin{equation*}
\begin{split}
\int_{-\frac12}^0 \int_{r/2}^r |du|^2 \, dV dt & \leq C \eps \left( \left( \frac{\lambda}{r} \right)^\nu  + r^\nu \right)^2.
\end{split}
\end{equation*}
Applying $\eps$-regularity (Theorem 3.4 of \cite{songwaldron}) again, we obtain
\begin{equation}\label{basecase:improvedestimate}
\begin{split}
r^{1 + k} |\nabla^{(k)} du(x,0)| & \leq C_k \sqrt{\eps} \left( \left( \frac{\lambda}{r} \right)^\nu  + r^\nu \right).
\end{split}
\end{equation}
To obtain the same estimate with $\nu = 1,$ one can apply a supersolution argument as in Propositions \ref{festimate}-\ref{gestimate}, with (\ref{basecase:improvedestimate}) in place of (\ref{basecase:crudestimate}). Since the sharp result will not be used below, we omit the proof. 
\end{proof}

\begin{thm}\label{thm:mainthm} 
Given $E > 0,$ $0 < \eps < \eps_0,$ and $0 < \alpha \leq 1,$ there exists $\delta_0 > 0$ as follows.  
Suppose that $u : B_R \times \left( -R^2, T \right) \to N$ is a smooth solution of (\ref{hm}), with $0 < R < R_0$ and $T > 0,$ which satisfies
\begin{equation}\label{mainthm:energyassumption}
\sup_{-R^2<t<T} \int_{B_R} |du(t)|^2 \, dV \leq E
\end{equation}
and
\begin{equation}\label{mainthm:deltaassumption}
\int_{-R^2}^T \!\! \int_{B_R} |\T|^2 \, dV dt < \delta_0^2.
\end{equation}
Suppose further that for some $0 < \lambda_1 \leq \frac12$ and $0 \leq t_1 < T,$ we have
\begin{equation}\label{mainthm:lambdaassumption}
\lambda_{\eps, R, x_0}(u(t)) \leq R \left( \lambda_1 + \left( \frac{t_1 - t}{R^2} \right)^{\frac{1 + \alpha}{2}} \right)
\end{equation}
for all $-R^2 < t \leq t_1.$ 
Then for $2 R \lambda_1 \leq r_g(x) \leq R / 2$ and each integer $k \geq 0,$ we have
\begin{equation}\label{mainthm:est}
r_g(x)^{1 + k} |\nabla^{(k)} du (x,t_1)| \leq C_{k,\alpha} \sqrt{\eps} \left( \frac{R\lambda_1}{r_g(x)} + \left( \frac{r_g(x)}{R} \right)^\alpha \right)^{\frac13 }.
\end{equation}
Here, $R_0 > 0$ depends on the geometry of $M,$ $\eps_0 > 0$ depends on the geometry of $N,$ and $\delta_0$ depends on $E,\eps,$ and $\alpha.$ 
\end{thm}
\begin{proof} 
For convenience, we replace (\ref{mainthm:lambdaassumption}) by
\begin{equation}\label{mainthm:kappathing}
\lambda_{\eps, R, x_0}(u(t)) \leq R \left( \lambda_1 + \kappa^{a} \left( \frac{t_1 - t}{R^2} \right)^{\frac{1 + \alpha}{2}} \right),
\end{equation}
where $0 < \kappa \leq \frac12$ is a constant 
to be determined by the argument below (depending only on $\alpha$ and $\nu$), and $a$ is sufficiently large, for instance
\begin{equation}\label{mainthm:athing}
a = \frac{72( 1 + \alpha) }{\alpha}.
\end{equation}
The assumptions (\ref{mainthm:kappathing}) and (\ref{mainthm:lambdaassumption}) are equivalent after rescaling and redefining constants. 

We will also prove the two estimates
\begin{equation}\label{mainthm:fhypoth}
 f(u;r,t_1) \leq C_0 \sqrt{\eps} \max \LB \frac{\lambda_1}{r}, r^{\alpha} \RB^{2 \nu - 1},
\end{equation}
\begin{equation}\label{mainthm:ghypoth}
g(u;r,t_1) \leq C_0 \sqrt{\eps} \max \LB \frac{\lambda_1}{r}, r^{\alpha} \RB^{\frac13},
\end{equation}
which clearly imply (\ref{mainthm:est}). Here, $\nu$ is any number with
\begin{equation}\label{mainthm:nucloseness}
\frac{26}{27} \leq \nu \leq \sqrt{1 - C \sqrt{\eps}} .
\end{equation}
It suffices to prove the theorem for $\lambda_1$ of the form
$$\lambda_1 = 2^{-n} \kappa^a,$$
so we may proceed by induction. For $\lambda_1 = \kappa^a,$
Lemma \ref{basecase} gives $\delta_0 = \delta > 0$ such that (\ref{mainthm:fhypoth}-\ref{mainthm:ghypoth}) hold, with $C_0 > 1$ universal. This establishes the base case. Note that we are free to assume $\kappa$ is arbitrarily small.


For the induction step, suppose that (\ref{mainthm:fhypoth}-\ref{mainthm:ghypoth}) hold for all $\lambda_1 \geq 2\bar{\lambda}_1,$ where $0 < \bar{\lambda}_1 \leq \kappa^a;$ {\it i.e.}, the conclusion of the theorem holds for all such $\lambda_1$ and solutions $u$ satisfying the hypotheses. We must establish the Theorem for $\lambda_1 = \bar{\lambda}_1.$

Let $u(t)$ be a solution satisfying the hypotheses, with $\lambda_1 = \bar{\lambda}_1.$ By rescaling, it suffices to assume $R = 1.$

In view of (\ref{mainthm:lambdaassumption}), 
the standard $\eps$-regularity lemma (see {\it e.g.} Theorem 3.4 of \cite{songwaldron}) implies
$$\sup_{\lambda(t) \leq r \leq R} \left( r |du| + r^2 |\nabla du| + r^3 |\nabla^2 du | + r^4 |\nabla^3 du |\right) \leq C \sqrt{\eps}.$$
Hence, by Lemma \ref{lemma:evolutions}, $f$ and $g$ satisfy evolution equations
\begin{equation*}
\begin{split}
\square_\nu f & \leq C \sqrt{\eps} \\
\square_\nu \left( \frac{g}{r} \right) & \leq \frac{6f}{r^3} + C \frac{\sqrt{\eps} }{r},
\end{split}
\end{equation*}
for $\lambda_1 + \kappa^{a} \left(t_1 - t \right)^{\frac{1 + \alpha}{2}} \leq r \leq 1.$ 
Let
\[
\begin{split}
\rho = 2 \bar{\lambda}_1, \qquad \zeta  = \rho^{\frac{1}{1 + \alpha}}.
\end{split}
\]
Since $\bar{\lambda}_1 \leq \kappa^a$ (by the base case), we have
\begin{equation}\label{mainthm:zetasmallness}
\frac{\rho}{\kappa} \leq \zeta \leq 2\kappa^{\frac{a}{1 + \alpha}}.
\end{equation}
Also notice that $\rho = \zeta^{1 + \alpha},$ so
\begin{equation}\label{mainthm:rhoandzeta}
\frac{\rho}{\zeta} = \zeta^\alpha.
\end{equation}
We now apply the induction hypotheses to $u$ with $R = 1/2,$ $\lambda_1 = \rho,$ and $R \rho = \bar{\lambda}_1.$ 
In view of (\ref{mainthm:rhoandzeta}), we clearly have (\ref{mainthm:fhypoth}-\ref{mainthm:ghypoth}) for $r \leq \zeta / 2.$ Applying the induction hypothesis again, with $R = 1$ and $\lambda_1 = \rho,$ we also obtain (\ref{mainthm:fhypoth}-\ref{mainthm:ghypoth}) for all $r \geq 2 \zeta.$ 
Hence, it remains only to establish (\ref{mainthm:fhypoth}-\ref{mainthm:ghypoth}) for 
\begin{equation}\label{mainthm:rrange}
\frac12 \zeta \leq r \leq 2 \zeta.
\end{equation}
In other words, for $r$ as in (\ref{mainthm:rrange}), we must show
\begin{equation}\label{mainthm:frrangetoshow}
f(r,t_1) \leq \frac{C_0}{2} \sqrt{\eps} \zeta^{\alpha (2 \nu -1)}
\end{equation}
and
\begin{equation}\label{mainthm:grrangetoshow}
g(r,t_1) \leq \frac{C_0}{2} \sqrt{\eps} \zeta^{\frac{\alpha}{3}}.
\end{equation}
Let
$$t_0 = t_1 - \frac{\zeta^2}{\kappa^2}.$$
From (\ref{mainthm:kappathing}), we have $\lambda(t) \leq \rho$ for all $t_0 \leq t \leq t_1.$ By the induction hypothesis, we have
\begin{equation}\label{mainthm:fappliedindhypoth}
f(r,t) \leq C \sqrt{\eps} \left( \frac{\rho}{r} + r^{\alpha} \right)^{2 \nu - 1}
\end{equation}
\begin{equation}\label{mainthm:gappliedindhypoth}
g(r,t) \leq C \sqrt{\eps} \left( \frac{\rho}{r} + r^{\alpha} \right)^{\frac13}
\end{equation}
for all $t_0 \leq t \leq t_1,$ where $C$ is a multiple of $C_0.$
Combining these, we also have
\begin{equation}\label{mainthm:energydecayhypoth}
r |du(x,t)| \leq C \sqrt{\eps} \left( \frac{\rho}{r} + r^{\alpha} \right)^{\frac13}
\end{equation}
for all $t_0 \leq t \leq t_1.$

To obtain the estimate (\ref{mainthm:frrangetoshow}) on $f,$ we apply Proposition \ref{festimate}. From (\ref{mainthm:fappliedindhypoth}-\ref{mainthm:gappliedindhypoth}), we obtain
\begin{equation*}
\begin{split}
\sup_{\zeta/2 \leq r \leq 2 \zeta} f(r,t_1) & \leq C C_{\ref{festimate}} \sqrt{\eps} \left( \kappa^{1 - \nu} \left( \frac{\rho}{\zeta} \right)^{2\nu - 1} + \kappa^{\nu - \alpha (2\nu - 1)} \zeta^{\alpha (2\nu - 1)} \right) \\
& \leq C \sqrt{\eps} \left(\kappa^{1 - \nu} + \kappa^{\nu - \alpha (2\nu - 1)} \right)\zeta^{\alpha (2\nu - 1)},
\end{split}
\end{equation*}
where we have used (\ref{mainthm:zetasmallness}-\ref{mainthm:rhoandzeta}).
Assuming that $\kappa$ is small enough that
$$C \sqrt{\eps} \left(\kappa^{1 - \nu} + \kappa^{\nu - \alpha (2\nu - 1)} \right) \leq \frac{C_0}{2},$$
this establishes the desired estimate (\ref{mainthm:frrangetoshow}) on $f.$\footnote{The decay estimate on $f$ can also be obtained directly from \cite{songwaldron}, Lemma 5.4. We have re-proven it here (by a different method) for the sake of exposition.}

Next, to obtain the estimate on $g,$ let
$$\rho_1 = \zeta^{1 + \frac{9}{16} \alpha}.$$
By the induction hypothesis, we have
\begin{equation*}
\sup_{t_0 \leq t \leq t_1} f(\rho_1, t) \leq C \sqrt{\eps} \left(\frac{\rho }{\rho_1 } \right)^{2\nu - 1} \leq C \sqrt{\eps}  \zeta^{\frac{7\cdot 25}{16 \cdot 27}\alpha } \leq C \sqrt{\eps} \zeta^{\frac{3}{8} \alpha}.
\end{equation*}
Using (\ref{mainthm:nucloseness}), we obtain
\begin{equation}\label{mainthm:ftimeest}
\int_{t_0}^{t_1} f(\rho_1, t)^2 \, dt 
\leq \frac{C \eps }{\kappa^2} \zeta^{2 + \frac{3}{4}\alpha }.
\end{equation}
We now integrate (\ref{stressidentity}) in time and apply H\"older's inequality:
\begin{equation}\label{mainthm:gfestimate1}
\begin{split}
\left| \int_{t_0}^{t_1} \left( g^2(\rho_1, t) - f^2(\rho_1, t) \right) \, dt \right| 
& = \left| \int_{t_0}^{t_1} \!\!\!\!\! \int_{S^1_{\rho_1}} X^i X^j S_{ij} (x,t) \, d\theta dt \right| \\
& \leq \left( \int_{t_0}^{t_1} \!\!\!\!\! \int_{D_{\rho_1}} |\T(u)|^2 \, dV dt  \right)^{\frac12} \left( \int_{t_0}^{t_1} \!\!\!\!\!  \int_{D_{\rho_1}} r'^2 \, |du|^2 \, dV_{r'} dt \right)^{\frac12} \\
& \leq C \delta_0 \left( (t_1 - t_0) \left( \rho^2 E + C \eps \rho_1^2 \left( \frac{\rho}{\rho_1} \right)^{\frac{2}{3} } \right) \right)^{\frac12}.
\end{split}
\end{equation}
Here we have used the assumptions (\ref{mainthm:energyassumption}-\ref{mainthm:deltaassumption}) and (\ref{mainthm:energydecayhypoth}).
We have
$$(t_1 - t_0) \rho^2 E = \frac{\zeta^{4 + 2 \alpha}}{\kappa^2} E$$
and 
$$(t_1 - t_0) \rho_1^2 \left( \frac{\rho}{\rho_1} \right)^{\frac{2}{3} \alpha} = \frac{\zeta^2}{\kappa^2} \zeta^{2 + \frac98 \alpha} \zeta^{\frac{7}{16} \frac23 \alpha} = \frac{ \zeta^{4 + \frac{17}{12} \alpha} }{\kappa^2}.$$
Hence, for $\delta_0$ sufficiently small (independently of $\bar{\lambda}_1$), (\ref{mainthm:gfestimate1}) reduces to
\begin{equation*}
\begin{split}
\left| \int_{t_0}^{t_1} \left( g^2(\rho_1, t) - f^2(\rho_1, t) \right) \, dt \right| 
& \leq \kappa \eps \zeta^{2 + \frac{17}{24} \alpha}.
\end{split}
\end{equation*}
Combining this with (\ref{mainthm:ftimeest}), we obtain
\begin{equation*}
\begin{split}
\int_{t_0}^{t_1} g(\rho_1, t)^2 \, dt & \leq \kappa \eps \zeta^{2 + \frac{17}{24} \alpha } + \frac{C \eps}{\kappa^2} \zeta^{2 + \frac{3}{4}\alpha } \\
& \leq C\eps  \left( \kappa + C \kappa^{\frac{1}{24}\frac{a \alpha}{1 + \alpha} - 2} \right) \zeta^{2 + \frac{17}{24} \alpha } \\
& \leq C \kappa \eps \zeta^{2 + \frac{17}{24}\alpha},
\end{split}
\end{equation*}
where we have used (\ref{mainthm:athing}) and (\ref{mainthm:zetasmallness}).

We may now apply Proposition \ref{gestimate}, to obtain
\begin{equation*}
\begin{split}
\sup_{\zeta/2 \leq r \leq 2 \zeta} g(r,t_1) & \leq C C_{\ref{gestimate}} \left( \begin{split}
& \frac{\sqrt{\kappa \eps} \zeta^{1 + \frac{17}{48} \alpha } }{\zeta^{1 + \frac{9}{16} \alpha} } \left( \frac{\zeta^{1 + \frac{9}{16}\alpha} }{\zeta} \right)^\nu + \sqrt{\eps} \left( \kappa^{\nu + \frac13 + 1} \left( \frac{\zeta^{1 + \alpha}}{\zeta}\right)^{\frac13} + \kappa^{\nu - \frac{\alpha}{3} + 1} \zeta^{\frac{\alpha}{3} } \right) \\
& \quad + \sqrt{\eps} \left( \frac{\zeta^{1 + \alpha} }{\zeta} + \zeta^{\alpha} \right)^{2 \nu - 1}
\end{split} \right) \\
& \leq C \left( \sqrt{\kappa \eps} \zeta^{\frac{\alpha}{48} \left( 17 + 27(\nu - 1) \right)} + \sqrt{\eps} \kappa \zeta^{ \frac{\alpha}{3} } + \sqrt{\eps} \zeta^{\frac{25}{27}\alpha} \right) \\
& \leq C \sqrt{\eps} \left( \sqrt{\kappa} + \kappa + \kappa^{\alpha \left( \frac{25a }{27(1 + \alpha)} - \frac13 \right)} \right) \zeta^{ \frac{\alpha}{3} },
\end{split}
\end{equation*}
where we have used (\ref{mainthm:nucloseness}) and (\ref{mainthm:zetasmallness}).
For $\kappa$ sufficiently small, this implies the desired bound (\ref{mainthm:grrangetoshow}), completing the induction.
\end{proof}

\begin{cor}[{\it Cf.} Theorem \ref{thm:intro}]\label{cor:detailedthm} Let $M$ be any Riemannian surface and suppose that $u$ is a classical solution of (\ref{hm}) on $M\times \LB 0, T \right)$
with bounded energy. Given $x_0 \in M$ and $0 < \eps < \eps_0,$ choose $0 < R < \min \LB R_0, \inj_{x_0}(M) \RB$ small enough that
\begin{equation}\label{lambdato0}
\qquad \lambda_{\eps, R, x_0} (u(t)) \to 0 \quad (t \nearrow T).
\end{equation}
If also
\begin{equation}\label{lambdaalphabound}
\lambda_{\eps, R, x_0} (u(t)) = O(T - t)^{\frac{1 + \alpha}{2}}, \qquad (t \nearrow T),
\end{equation}
where $0 < \alpha \leq 1,$ then $u(T)$ is $C^{\frac{\alpha}{3} }$ on $B_{R/2}(x_0).$
\end{cor}
\begin{proof} First note that since the energy is bounded, the stress-energy tensor is also bounded in $L^1.$ By Corollary 4.5 of \cite{songwaldron}, with $q = 1,$ (\ref{lambdato0}) is true for $R > 0$ sufficiently small.

Now, given any $\lambda_1 > 0,$ in view of (\ref{lambdaalphabound}), (\ref{mainthm:lambdaassumption}) will be satisfied for all $t_1$ sufficiently close to $T.$  
For $x \in U_{2\lambda_1}^R(x_0),$ by Theorem \ref{thm:mainthm}, we have
\begin{equation*}
r|du(x,T)| = \lim_{t_1 \nearrow T} r|du(x,t_1)| \leq C \sqrt{\eps} \left( \frac{R\lambda_1}{r} + \left( \frac{r}{R} \right)^\alpha \right)^{\frac13}.
\end{equation*}
Letting $\lambda_1 \searrow 0,$ we obtain
\begin{equation*}
r|du(x,T)| \leq C\sqrt{\eps} \left( \frac{r}{R} \right)^{ \frac{\alpha}{3}}
\end{equation*}
for all $x \in B_{R/2} \setminus \{x_0\}.$
Integrating radially, this gives $u(T) \in C^{\frac{\alpha}{3}}.$ 
\end{proof}

\appendix

\section{Proof of Lemma \ref{lemma:evolutions}}\label{app:evolutions}

Let
\begin{equation*}
f_0(r) = \sqrt{ \int |u_\th (r,\theta,t)|^2 d\th}, \qquad f_1(r) = \sqrt{ \int |\nabla_\theta u_\th (r,\theta,t)|^2 d\th }.
\end{equation*}
In our previous paper \cite{songwaldron}, Lemma 5.1, we calculated the desired evolution of $f_0;$ we now apply a similar analysis to $f_1.$ These calculations go back to Lin-Wang \cite{linwangenergyidentity} and Parker \cite{parkerhmbubbletree}.

For convenience, we shall work below in cylindrical coordinates $(s=\ln r,\th).$ Letting
$$g_0=ds^2+d\th^2$$
denote the flat cylindrical metric, we have $g=\xi^{2} e^{2s} g_0.$
The differential of $u$ is given by
\begin{equation*}
du = u_s ds + u_\theta d\theta.
\end{equation*}
The tension field with respect to $g_0$ is given by
\[ \mathcal{T}_0(u)=\n_s u_s+\n_\th u_\th,\]
where $\n$ denotes the pullback connection on $u^*TN,$ as above.
The heat-flow equation (\ref{hm}) with respect to the metric $g$ becomes
\begin{equation}\label{e:hm-cylin}
u_t = \mathcal{T}(u) = \xi^{-2}e^{-2s}\mathcal{T}_0(u).
\end{equation}
We start from the identity
\begin{equation}\label{ds2f12}
\frac12 \p_s^2f_1^2= \int_{S^1}\( |\n_s \nabla_\theta u_{\th}|^2+\< \n_s^2 \nabla_\theta u_{\th}, \nabla_\theta u_\th\>\).
\end{equation}
We have
\begin{equation}\label{T0ut}
\nabla_s u_s + \nabla_\theta u_\theta = \T_0(u) =  \xi^2 e^{2s} u_t.
\end{equation}
Applying $\n_\theta,$ we obtain
\begin{equation*}
\begin{split}
\nabla_\theta \nabla_s u_s + \nabla^2_\theta u_\theta & = 2 \xi \nabla_\theta \xi e^{2s} u_t + \xi^2 e^{2s} \nabla_\theta u_t \\
& =2 \xi \nabla_\theta \xi e^{2s} u_t + \xi^2 e^{2s} \nabla_t u_\theta,
\end{split}
\end{equation*}
and
\begin{equation}\label{dtheta2sus}
\begin{split}
\nabla_\theta^2 \nabla_s u_s & = - \nabla^3_\theta u_\theta + \left( 2 (\nabla_\theta \xi)^2 + 2 \xi \nabla^2_\theta \xi \right) e^{2s} u_t \\
& \quad + 2 \xi \nabla_\theta \xi e^{2s} \nabla_t u_\theta + \xi^2 e^{2s} \nabla_\theta \nabla_t u_\theta \\
& = - \nabla^3_\theta u_\theta + e^{2s} \nabla_t \nabla_\theta u_\theta + I,
\end{split}
\end{equation}
where
\begin{equation}\label{I}
I = (\xi^2 - 1) e^{2s} \nabla_\theta^2 u_t + e^{2s} R(u_\theta, u_t) u_\theta + \left( 2 (\nabla_\theta \xi)^2 + 2 \xi \nabla^2_\theta \xi \right) e^{2s} u_t + 2 \xi \nabla_\theta \xi e^{2s} \nabla_\theta u_t.
\end{equation}
We may also commute derivatives to obtain
\begin{equation}\label{ds2thetautheta}
\begin{split}
\nabla_s^2 \nabla_\theta u_\theta & = \nabla_s \left( \nabla_\theta \nabla_s u_\theta + R(u_s, u_\theta) u_\theta \right) \\
& = \nabla_\theta \nabla_s^2 u_\theta + \nabla_s \left( R(u_s, u_\theta) u_\theta \right) \\
& = \nabla_\theta \nabla_s ( \nabla_\theta u_s ) + \nabla_s \left( R(u_s, u_\theta) u_\theta \right) \\
& = \nabla_\theta^2 \nabla_s u_s + I \! I,
\end{split}
\end{equation}
where
\begin{equation}\label{II}
\begin{split}
I\! I & = \nabla_\theta \left( R(u_s, u_\theta) u_s \right) + \nabla_s \left( R(u_s, u_\theta) u_\theta \right) \\
& = \nabla R \left( u_\theta, u_s, u_\theta \right) u_s + R(\nabla_\theta u_s, u_\theta) u_s +R( u_s, \nabla_\theta u_\theta) u_s + R(u_s, u_\theta) \nabla_\theta u_s \\
& \quad + \nabla R \left( u_s, u_s, u_\theta \right) u_\theta + R(\nabla_s u_s, u_\theta) u_\theta + R( u_s, \nabla_s u_\theta) u_\theta + R( u_s, u_\theta) \nabla_s u_\theta.
\end{split}
\end{equation}
Inserting (\ref{dtheta2sus}) and (\ref{ds2thetautheta}) into (\ref{ds2f12}), integrating by parts, and rearranging, we obtain
\begin{equation}\label{firstevol}
\frac12 \left( e^{2s} \partial_t f_1^2 - \p_s^2f_1^2 \right) = - \int_{S^1}\( |\n_s \nabla_\theta u_{\th}|^2 + |\n_\theta^2 u_\theta|^2 \right) - \int_{S^1} \LA I + I\! I, \nabla_\theta u_\theta \RA.
\end{equation}
We need the following simple estimates. Since $|u_\th|\le \eta$ is small, we may assume that the image of the curve $u(s, \cdot, t):S^1\to N$ lies in a coordinate chart of $N$ where the Christoffel symbol $\Ga$ is bounded by $C_N$. Then, in local coordinates, we have
\[ |\p_\th u_\th| = |\n_\th u_\th - \Ga(u_\th, u_\th)| \le |\n_\th u_\th| + |\Ga(u_\th, \n_\th u_\th)| \le \eta + C_N \eta^2 \leq 2 \eta,\]
assuming that $\eta$ is sufficiently small (depending on $N$).
Thus
\[\begin{aligned}
  |\n_{\th} u_{\th}|^2&=|\p_\th u_\th +\Ga(u_\th, u_\th)|^2\\
  &\ge   |\p_\th u_\th|^2-2|\p_\th u_\th||\Ga(u_\th, u_\th)|+|\Ga(u_\th, u_\th)|^2\\
  &\ge  |\p_\th u_\th|^2-C \eta |u_\th|^2.
\end{aligned}\]
Then the ordinary Poincar\'e inequality on $S^1$ yields
\begin{equation}\label{f1overf_0}
f_1^2 =  \int_{S^1}|\n_{\th} u_{\th}|^2\ge  \int_{S^1}|u_\th|^2-C \eta \int_{S^1}|u_\th|^2=(1-C\eta )f_0^2.
\end{equation}
A similar argument, applied to $\nabla_\theta^2 u_\theta,$ gives
\begin{equation}\label{nablaf1overf1}
\int_{S^1}|\n_{\th}^2 u_{\th}|^2 \ge  (1-C\eta) f_1^2.
\end{equation}
We first apply (\ref{nablaf1overf1}) and H\"older's inequality to (\ref{firstevol}), to obtain
\begin{equation}\label{secondevol}
\begin{split}
\frac12 \left( e^{2s} \partial_t f_1^2 - \p_s^2f_1^2 + (1 - C\eta) f_1^2 \right) & \leq - \int_{S^1} |\n_s \nabla_\theta u_{\th}|^2 - \int_{S^1} \LA I + II, \nabla_\theta u_\theta \RA \\
& \leq - \int_{S^1} |\n_s \nabla_\theta u_{\th}|^2 + \left( \|I\|_{L^2(S^1)} + \|II\|_{L^2(S^1)} \right) f_1.
\end{split}
\end{equation}
Note that
\[ \frac12\p_s f_1^2= f_1\p_s f_1=\int_{S^1}\<\n_s \n_\th u_{\th}, \n_\th u_\th\>\le \(\int_{S^1}|\n_s \n_\th u_{\th}|^2\)^{1/2}f_1,\]
so we have
$$|\p_s f_1|^2\le\int |\n_s \n_\th u_{\th}|^2.$$
On the other hand,
\[ \frac12\p_s^2(f_1^2)=f_1\cdot \p_s^2 f_1+ |\p_s f_1|^2.\]
Hence, ``dividing out'' by $f_1$ in (\ref{secondevol}) (which is justified in the distribution sense), we obtain
\begin{equation}\label{thirdevol}
\begin{split}
e^{2s} \partial_t f_1 - \p_s^2f_1 + (1 - C\eta) f_1 & \leq 2 \left( \|I\|_{L^2(S^1)} + \|II\|_{L^2(S^1)} \right).
\end{split}
\end{equation}
It remains to estimate the terms on the RHS of (\ref{thirdevol}). By (\ref{evolutions:derivbounds}), we have
\begin{equation*}
e^{2s} \left(|u_t| + |\nabla_\theta u_t| + |\nabla^2_\theta u_t| \right) \leq C \eta.
\end{equation*}
Combining this with (\ref{smallness1}), from (\ref{I}), we obtain
\begin{equation*}
\left\| I \right\|_{L^2(S^1) } \leq C \xi_0 \eta e^{2s} + \eta^2 f_0 \leq C \xi_0 \eta e^{2s} + \eta f_1,
\end{equation*}
where we have used (\ref{f1overf_0}).
From (\ref{II}), since each term has at least one factor of $u_\theta$ or $\n_\theta u_\th,$ we also obtain
\begin{equation*}
\left\| I\! I \right\|_{L^2(S^1) } \leq C_N \eta^2 f_1 \leq \eta f_1,
\end{equation*}
for $\eta$ sufficiently small.
Inserting these estimates into (\ref{thirdevol}), and absorbing the $C \eta f_1$ terms in the LHS, we obtain
\begin{equation}\label{fourthevol}
\begin{split}
e^{2s} \partial_t f_1 - \p_s^2f_1 + (1 - C\eta) f_1 & \leq C \xi_0 \eta e^{2s}.
\end{split}
\end{equation}
Translating the above equation back to polar coordinates, we get (\ref{evolutions:angular}).

To estimate the radial energy
$$g = \sqrt{ \int_{ \{e^s\} \times S^1 } |u_s|^2 d\th },$$
we start from the identity
\begin{equation}\label{urevol:1}
 \frac12 \p_s^2 g^2= \int_{S^1}\( |\n_s u_s|^2+\< \n_s^2u_{s}, u_s\>\).
\end{equation}
Applying $\n_s$ to (\ref{T0ut}), we obtain
\begin{equation*}
\begin{split}
\nabla^2_s u_s + \nabla_s \nabla_\theta u_\theta & = \left( 2 \xi \partial_s \xi + 2 \xi^2 \right) e^{2s}  u_t + \xi^2 e^{2s} \nabla_s u_t \\
& = 2 \left( \xi^{-1} \partial_s \xi + 1 \right) \T_0(u) + \xi^2 e^{2s} \nabla_t u_s \\
& = 2 \left( \nabla_s u_s + \nabla_\theta u_\theta \right) + e^{2s} \nabla_t u_s + 2 \xi^{-1} \partial_s \xi (\nabla_s u_s + \nabla_\theta u_\theta)+ \left( \xi^2 - 1 \right) e^{2s} \nabla_s u_t.
\end{split}
\end{equation*}
We also have
$$\nabla_s \nabla_\theta u_\theta = \nabla_\theta^2 u_s + R(u_s, u_\theta) u_\theta.$$
Returning to (\ref{urevol:1}), we have
\begin{equation*}
\begin{split}
\frac12 \partial_s^2 g^2 & = \int |\nabla_s u_s|^2 + \int \LA e^{2s} \nabla_t u_s + 2 \nabla_s u_s + 2 \nabla_\theta u_\theta , u_s \RA \\
& \qquad + \int |\nabla_\theta u_s |^2 + \int \LA 2 \xi^{-1} \partial_s \xi \left( \nabla_s u_s + \nabla_\theta u_\theta \right) + (\xi^2 - 1) e^{2s} \nabla_s u_t , u_s \RA - \int R(u_s, u_\theta) u_\theta,
\end{split}
\end{equation*}
where we have integrated by parts once.
Applying H\"older's inequality, rearranging, and using (\ref{smallness1}), we obtain
\begin{equation*}
\begin{split}
\frac12 \left( e^{2s} \partial_t - \partial_s^2 + 2 \partial_s \right) g^2 & \leq - \int |\nabla_s u_s|^2  + 2 f_1 g + C \xi_0  \eta e^{2s} g \\
& \qquad - \int |\nabla_\theta u_s |^2 + C_N \eta f_0 g.
\end{split}
\end{equation*}
We choose $\eta$ small enough that $C_N \eta \leq 1,$ and discard the $-\int |\nabla_\theta u_s|^2$ term. After ``dividing out'' by $g$ as above, we obtain
\begin{equation}\label{urevol:2}
\begin{split}
\left( e^{2s} \partial_t - \partial_s^2 + 2 \partial_s \right) g \leq 6f_1 + C \xi_0 \eta e^{2s}.
\end{split}
\end{equation}
Changing back to polar coordinates and dividing by $r,$ we get the desired evolution equation.

\section{Radial heat kernel}

In this appendix, we extract several results from the appendix of \cite{lte}, replacing the integer dimension by a real number $\mu > 1.$ The proofs of Propositions \ref{festimate}-\ref{gestimate} are based on these results.

Let
$$\Delta_\mu = \p_r^2 + \frac{\mu - 1}{r} \p_r.$$
In the case that $\mu$ is an integer, the spherical average of the Euclidean heat kernel is given by
\begin{equation}\label{Halphadef}
H(r,s,t) = \frac{c_\mu e^{\frac{-(r^2 + s^2)}{4t}}}{t^{\mu/2} } I\left( \frac{rs}{2t} \right),
\end{equation}
where $c_\mu$ is an appropriate constant, and\footnote{Since $I$ satisfies the ODE (\ref{Iode}), we in fact have
$$I(x) = x^{1 - \frac{\mu}{2}} I_{\frac{\mu}{2} - 1}(x),$$
where $I_{\frac{\mu}{2} -1}$ is the modified Bessel function. This recovers formula (3.3) of Bragg \cite{braggradialheatpolys}.}
\begin{equation}\label{varphidef}
I(x) = \int_0^\pi e^{x \cos \theta } \sin^{\mu - 2} \theta \, d\theta.
\end{equation}

\begin{lemma}\label{lemma:euclheatker} For any real $\mu > 1,$ the above function $H$ satisfies
\begin{equation*}
\begin{split}
\left( \p_t - \Delta_\mu \right) H (\cdot, s,t) = 0, \qquad t > 0, \\
H (r,s,t) > 0 \mbox{ for } 0 < r,s,t < \infty, \\
H (r, s,t) \to \frac{1}{s^{\mu - 1}} \delta(r-s) \qquad t \searrow 0,
\end{split}
\end{equation*}
and
\begin{equation}\label{Hbound}
\frac{C_\mu^{-1} e^{\frac{-(r-s)^2}{4t}}}{t^{1/2} (rs + t)^{\frac{n-1}{2}} } \leq H(r,s,t) \leq \frac{C_\mu e^{\frac{-(r-s)^2}{4t}}}{t^{1/2} (rs + t)^{\frac{n-1}{2}} }.
\end{equation}
\end{lemma}
\begin{proof} 
We calculate
\begin{equation*}
\left( \p_t - \Delta_\mu \right) H(r,s,t) = \frac{c_\mu e^{\frac{-(r^2 + s^2)}{4t}}}{t^{\mu/2} } \left( \frac{-s^2}{4t^2} \right) \left( I''(x) + \frac{(\mu - 1)}{x} I'(x) - I(x) \right).
\end{equation*}
For $I(x)$ given by (\ref{varphidef}), we have
\begin{equation}\label{Iode}
\begin{split}
I''(x) + \frac{(\mu-1)}{x} I'(x) - I(x) & = \int_0^\pi e^{x \cos \theta} \left( \cos^2 \theta + \frac{(\mu-1)}{x}\cos\theta - 1 \right) \sin^{\mu - 2} \theta \, d\theta \\
& = \int_0^\pi e^{x \cos \theta} \left( -\sin^2 \theta + \frac{(\mu - 1)}{x}\cos\theta\right) \sin^{\mu - 2} \theta \, d\theta \\
& = 0,
\end{split}
\end{equation}
after integrating by parts.
Hence, $H$ solves the PDE as required for any real $\mu.$

Borrowing a factor of $e^{\frac{rs}{2t}}$ in (\ref{varphidef}), we have
\begin{equation}\label{HI1}
H(r,s,t) = \frac{c_\mu e^{\frac{-(r - s)^2}{4t}}}{t^{\mu/2} } I_1\left( \frac{rs}{2t} \right),
\end{equation}
where
\begin{equation*}
I_1(x) = \int_0^\pi e^{x (\cos \theta - 1) } \sin^{\mu - 2} \theta d\theta.
\end{equation*}
Then $I_1(x)$ clearly tends to a positive constant as $x \to 0.$ By the substitution $u = \sqrt{x(1 - \cos \theta )},$ it follows that the integral is bounded by a constant times $x^{-\frac{\mu - 1}{2}}.$ Hence
$$I_1(x) \leq \frac{C}{(1 + x)^{\frac{\mu - 1}{2}}} \leq \left( \frac{t}{rs + t} \right)^{\frac{\mu - 1}{2}}.$$
Substituting into (\ref{HI1}), we obtain the desired bound.
\end{proof}

\subsection{Initial data}\label{app:initialkernel} Let $H_{\LB \rho, R \RB}(r,s,t)$ be the Dirichlet kernel for the operator $\p_t - \Delta_\mu$ on the interval $\LB \rho, R\RB,$ satisfying
\begin{align*}
\left( \p_t - \Delta_\mu \right) H_{\LB \rho, R\RB}(\cdot, s,t) & = 0 & t > 0, \\
H_{\LB \rho, R\RB}(\rho, s,t) & = 0 = H_{\LB \rho, R\RB} (R, s,t) &  \rho \leq s \leq R , \,\, t > 0, \\
H_{\LB \rho, R\RB}(r, s,t) & \to \frac{1}{s^{\mu - 1}} \delta(r - s) &  t \searrow 0.
\end{align*}
By the maximum principle, we have $0 \leq H_{\LB \rho, R \RB}(r,s,t) \leq H(r,s,t),$ and so
\begin{equation}\label{HrhoRbound}
0 \leq H_{\LB \rho , R \RB} (r,s,t) \leq \frac{C_\mu e^{\frac{-(r-s)^2}{4t}}}{t^{1/2} (rs + t)^{\frac{\mu-1}{2}} }.
\end{equation}
Given an initial function $\varphi(r)$ on $\LB \rho, R \RB,$ the solution of the initial-value problem is given by
\begin{equation}\label{ivp}
v_0(r,t) = \int_\rho^R H_{\LB \rho , R \RB} (r,s,t) \varphi(s) s^{\mu - 1} \, ds.
\end{equation}
Let
$$w^a(r,t) = \left( \frac{r^2}{r^2 + t} \right)^{a/2}.$$

\begin{prop}\label{initialestimate} For $0 \leq k \leq \mu - 1,$ assuming that $|\varphi(r) | \leq A r^{-k},$ we have
\begin{equation*}
|v_0(r,t)| \leq C_\mu A r^{-k} w^{k}(r,t) w^{\mu - k}(R,t).
\end{equation*}
\end{prop}
\begin{proof} From (\ref{HrhoRbound}) and (\ref{ivp}), we have
\begin{equation*}
\begin{split}
|v_0(r,t)| & \leq C A \int_\rho^R e^{-(r-s)^2/4t} \frac{s^{-k + \mu -1}}{(rs + t)^{\frac{\mu -1}{2}}} \, \frac{ds}{t^{1/2}} \\
& \leq CA r^{-k} \int_\rho^R e^{-(r-s)^2/4t} \frac{r^{k} s^{-k + \mu - 1}}{(rs + t)^{\frac{\mu -1}{2}}} \, \frac{ds}{t^{1/2}}.
\end{split}
\end{equation*}
By Lemma A.1a of \cite{lte}, applied with $a = k, b= \mu - k -1,$ and $c = d = 0,$ we have
\begin{equation*}
\begin{split}
\int_\rho^R e^{-(r-s)^2/4t} \frac{r^{k} s^{\mu - k - 1}}{(rs + t)^{\frac{\mu -1}{2}}} \, \frac{ds}{t^{1/2}} & \leq C \frac{ R - \rho}{R - \rho + \sqrt{t} } w^{k}(r,t) w^{\mu - k - 1}(R,t) \\
& \leq C w^{k}(r,t) w^{\mu - k}(R,t).
\end{split}
\end{equation*}
The result follows.
\end{proof}


\subsection{Boundary data}\label{app:innerkernel} To construct a kernel for the boundary data at the inner radius $\rho = 1,$ we follow the argument of \cite{lte}, Appendix A.3. Suppose $R > 1,$ and let
$$h(r)  = \frac{r^{2 - \mu} - R^{2 - \mu} }{1 - R^{2 - \mu} }.$$
Let
$$y_1(r,t) = h(r) - \int_1^R H_{\LB 1, R \RB}(r,s,t) h(s) \, s^{\mu - 1} ds.$$
This satisfies
\begin{equation*}
\begin{split}
(\p_t - \Delta_\mu) y_1 & = 0 \\
y_1(r,0) & = 0, \quad 1 < r < R \\
y_1(1,t) & = 1, \,  y_1(R,t) = 0, \quad t > 0.
\end{split}
\end{equation*}
The function
$$G_{\LB 1, R\RB}(r,t) = \p_t y_1(r,t)$$
satisfies
$$\lim_{r \searrow 1} G_{\LB 1, R\RB}(r,t) = \delta(t).$$
\begin{lemma}\label{boundarykernel}
We have
\begin{equation*}
0 \leq G_{\LB 1, R\RB}(r,t) \leq \frac{C_{\mu} e^{\frac{-(r -1)^2}{5t} }}{t(t + 1)^{\frac{\mu}{2} - 1}} \cdot \begin{cases} \min \LB (r-1)/\sqrt{t}, 1 \RB & (t \leq 1) \\
\min \LB r - 1, 1 \RB & (t \geq 1). \end{cases}
\end{equation*}
\end{lemma}
\begin{proof} Replacing $n$ by $\mu,$ the bound is identical to that of Lemma A.4b of \cite{lte}, and the proof there carries over. 
\end{proof}

To obtain an inner boundary kernel for $\LB \rho, R \RB,$ we let
$$G_{\LB \rho, R \RB}(r,t) = \frac{1}{\rho^{2}} G_{\LB 1, R/\rho\RB}(r/\rho,t/\rho^2).$$
By Lemma \ref{boundarykernel}, this satisfies
\begin{equation}\label{Gbound}
G_{\LB \rho, R \RB}(r,t) \leq \frac{C_{\mu} e^{\frac{-(r -\rho)^2}{5t} } \rho^{\mu - 2} }{t(t + \rho^2 )^{\frac{\mu}{2} - 1}} \cdot \begin{cases} \min \LB (r-\rho)/\sqrt{t}, 1 \RB & (t \leq 1) \\
\min \LB r - \rho, 1 \RB & (t \geq 1). \end{cases}
\end{equation}
The solution of the boundary problem with data $\psi(t)$ at $r = \rho$ is given by
\begin{equation}\label{boundaryivp}
v_1(r,t) = \int_0^t \psi(\tau) G_{\LB \rho, R \RB}(r,t - \tau) \, d\tau.
\end{equation}

\begin{prop}\label{innerestimate} For $2\rho \leq r \leq R$ and $t \geq 0,$ we have
\begin{equation*}
| v_1(r,t) | \leq C_\mu e^{-(r - \rho)^2/6t} \left( \frac{\rho^{ \mu - 2 }}{r^{\mu - 1}} \right) \sqrt{ \int_0^t \psi^2(\tau) d\tau }.
\end{equation*}
\end{prop}
\begin{proof} We apply H\"older's inequality as in the proof of Proposition A.5b of \cite{lte}. From (\ref{Gbound}-\ref{boundaryivp}), we have:
\begin{equation*}
\begin{split}
r^{\mu - 1} |v_1(r,t)| & \leq C_\mu \rho^{ \mu  - 2} \int_0^t |\psi(\tau)| \frac{ r^{\mu - 1} e^{\frac{-(r -\rho)^2}{5(t - \tau)} }}{(t - \tau)(t - \tau + \rho^2)^{\frac{\mu}{2} - 1}} \, d\tau \\
& \leq C_\mu \rho^{ \mu - 2} \sqrt{\int_0^t \psi^2(\tau) \, d\tau } \\
& \qquad \cdot \sqrt{ \int_0^t \frac{r^{2\mu-2}(r - \rho)^2 }{(t - \tau)^2(t - \tau + \rho^2)^{\mu - 2} } e^{\frac{-2(r -\rho)^2}{5(t - \tau)} } \, \frac{d\tau}{(r - \rho)^2} }.
\end{split}
\end{equation*}
The result follows by changing variables $u = \frac{\tau}{(r  - \rho)^2}.$ 
\end{proof}

\bibliographystyle{amsinitial}
\bibliography{biblio}

\end{document}